\documentclass[12pt,reqno,intlim]{amsart}

\usepackage{amssymb,amsmath}

\usepackage[arrow,matrix,curve]{xy}

\newdir{ >}{{}*!/-5pt/\dir{>}}

\newcommand{\nc}{\newcommand}

\nc{\bC}{\bold{C}} \nc{\bN}{\Bbb{N}} \nc{\cF}{\mathcal{F}}
\nc{\cE}{\mathcal{E}} \nc{\cR}{\mathcal{R}} \nc{\cM}{\mathcal{M}}
\nc{\al}{\alpha} \nc{\bt}{\beta} \nc{\gm}{\gamma} \nc{\dl}{\delta}
\nc{\om}{\omega} \nc{\sg}{\sigma} \nc{\Sg}{\Sigma} \nc{\vf}{\varphi}
\nc{\ve}{\varepsilon} \nc{\os}{\overset} \nc{\ol}{\overline}
\nc{\ul}{\underline} \nc{\us}{\underset} \nc{\sbs}{\subset}
\nc{\bsl}{\backslash} \nc{\Ra}{\Rightarrow}
\nc{\lra}{\longrightarrow} \nc{\all}{\allowdisplaybreaks}

\nc{\Codes}{\operatorname{{\bold{Codes}}}}
\nc{\RegMono}{\operatorname{\mathcal{R}{\rm{eg}\mathcal{M}{\rm{ono}\!}}}}
\nc{\RegEpi}{\operatorname{\mathcal{R}{\rm{eg}\mathcal{E}{\rm{pi}\!}}}}
\nc{\Mn}{\operatorname{\mathcal{M}{\rm{ono}\!}}}
\nc{\Ep}{\operatorname{\mathcal{E}{\rm{pi}\!}}}
\nc{\Rg}{\operatorname{\mathcal{R}{\rm{eg}\!}}}
\nc{\Ob}{\operatorname{Ob\!}}

\numberwithin{equation}{section}

\newtheorem{theo}{\ \ \ Theorem}[section]
\newtheorem{lem}[theo]{\ \ \ Lemma}
\newtheorem{prop}[theo]{\ \ \ Proposition}
\newtheorem{cor}[theo]{\ \ \ Corollary}

\theoremstyle{definition}

\theoremstyle{remark}

\begin{document}

\title[]
{On sums of powers of natural numbers}

\author{Eteri Samsonadze}

\maketitle

 \begin{abstract}
The problem of finding the sum of a polynomial's values is considered. In particular, for any $n\geq 3$, the explicit formula for the sum of the $n$th powers of natural numbers $S_n=\sum_{x=1}^{m}x^{n}$ is proved:
 $$\sum_{x=1}^{m}x^{n}=(-1)^{n}m(m+1)(-\frac{1}{2}+\sum_{i=2}^{n}a_i(m+2)(m+3)...(m+i)),$$
here $a_i=\frac{1}{i+1}\sum_{k=1}^{i}\frac{(-1)^{k}k^{n}}{k!(i-k)!}$, $(i=2,3,...,n-1)$, $a_n=\frac{(-1)^n}{n+1}$. Note that this formula does not contain Bernoulli numbers.

 \bigskip


\end{abstract}
 
 \section{Introduction}
For $n=1, 2, 3, 4$, the explicit formulas for the sum of powers of natural numbers  
$$S_n=\sum_{x=1}^{m}x^{n}$$ 
were well-known already in the ancient times. In 1613, Faulhaber published a paper, where some properties of the function  $S_n(m)$ were studied and, moreover, the explicit formulas for high odd values of $n$ were given \cite{F}. Nowadays, for calculating $S_n$, there are various recurrent formulas, geometric and matrix methods, as well as formulas that are valid for arbitrary $n$, but contain Bernoulli numbers (see, e.g., \cite{BV}, \cite{E}, \cite{Ge}, \cite{Ka}, \cite{MS}, \cite{Ti}, \cite{T}, \cite{W}, \cite{K}). In \cite{GR}, the explicit formulas for $S_n$ are given for the cases $n=1, 2,...,7$.

  
  In the present work,  we give an explicit formula for $S_n$, for an arbitrary natural $n$; this formula does not contain Bernoulli numbers. To this end, we first find the explicit formula for the function $$g(m)=\sum_{x=1}^{m}f(x),$$ where $f(x)$ is a polynomial. The found formula implies that $g(m)$ is a polynomial and is divisible by $m$.
  
   Applying again the found formula for $g(m)$, the explicit formula for the sum of the $n$th powers of natural numbers is obtained:
 \begin{equation}
 \sum_{x=1}^{m}x^{n}=(-1)^{n}\sum_{i=1}^{n}a_im(m+1)(m+2)...(m+i),
 \end{equation}
where  
\begin{equation}
a_i=\frac{1}{i+1}\sum_{k=1}^{i}\frac{(-1)^{k}k^{n}}{k!(i-k)!},
\end{equation}
 $(i=1, 2,...,n)$.  Note that this formula does not contain Bernoulli numbers. Note also that, for the same sum, the following formula was given without proof in our earlier paper \cite{S}:
$$ \sum_{x=1}^{m}x^{n}=\sum_{i=1}^{n}\sum_{k=1}^{i}\frac{(-1)^{k+n}k^{n}m(m+1)(m+2)...(m+i)C_i^k}{(i+1)!},$$
where $C_i^k$ denotes the binomial coefficient ${\begin{pmatrix}i\\k\\\end{pmatrix}}$.
 
 Taking the fact that the principal term of the polynomial $$S_n(m)=\sum_{x=1}^{m}x^{n}$$ is $\frac{1}{n+1}m^{n+1}$ \cite{K}, \cite{BSh}, the calculation of $a_n$ can be simplified. Namely, we obtain that 
 \begin{equation}
 a_n=\frac{(-1)^n}{n+1}.
 \end{equation}
  The above-mentioned formulas  imply that if $n\geq 3$, then 
$$\sum_{x=1}^{m}x^{n}=(-1)^{n}m(m+1)(-\frac{1}{2}+\sum_{i=2}^{n}a_i(m+2)(m+3)...(m+i)),$$
where $a_i$ is given by (1.2), for $i=2,3,...,n-1$, while $a_n$ is given by (1.3). 


\section{The sum of powers of natural numbers}

Below we give the formula for the sum of the values of a polynomial $f(x)$ in the points $x=1,2,...,m$. To this end, we first give the following lemma.

\begin{lem}
If the degree of a polynomial $f(x)$ is less than or equal to $n$, then the following equality holds:
\begin{equation}
f(x)=f(0)+ \sum_{i=1}^{n}d_ix(x+1)(x+2)...(x+i-1),
\end{equation}
where 
\begin{equation}
d_i=\sum_{k=o}^{i}\frac{(-1)^kf(-k)}{k!(i-k)!} \; \; \; \; (i=1,2,...,n).
\end{equation}
\end{lem}

\begin{proof}
Consider the polynomial
\begin{equation}
h(x)=l_0+\sum_{i=1}^{n}l_ix(x+1)(x+2)...(x+i-1),
\end{equation}
and choose coefficients $l_0$, $l_1$, ..., $l_n$ in such a way that the values of the polynomials $f(x)$ and $h(x)$ coincide in the $(n+1)$ points:
$$f(-j)=h(-j)\; \; \; \; (j=0,1,2,...,n).$$
Then $$f(x)\equiv h(x),$$ whence 
\begin{equation}
f(-j)=l_0-l_1j+l_2j(j-1)+...+(-1)^jl_jj(j-1)(j-2)\cdot ... \cdot 2\cdot 1 .
\end{equation}
$(j=0,1,2,...,n)$.
Applying the well-known formula
$$C_k^0-C_k^1+C_k^2+...+(-1)^kC_k^k=0,$$
($k\in \mathbb{N}$) where $C_i^j$ denotes the binomial coefficient ${\begin{pmatrix}j\\i\\\end{pmatrix}}$ , it is not hard to verify that the solution of the system (2.4) is 
\vskip+2mm
$$l_0=f(0), \; \; \; l_i=\sum_{k=0}^{i}\frac{(-1)^kf(-k)}{k!(i-k)!} \; \; \; (i=1,2,...,n).$$
\vskip+2mm
This implies formulas (2.1) and (2.2).
\end{proof}
\vskip+3mm
Consider the function 
$$g(m)=\sum_{x=1}^{m}f(x).$$
As is well-known \cite{V}, 
\begin{equation}
\sum_{x=1}^{m}x(x+1)(x+2)...(x+i-1)=\frac{1}{i+1}m(m+1)(m+2)...(m+i).
\end{equation}
Therefore, 
Lemma 2.1 implies
\begin{theo}
\begin{equation}
\sum_{x=1}^{m}f(x)=mf(0)+\sum_{i=1}^{n}\frac{d_i}{i+1}m(m+1)(m+2)...(m+i),
\end{equation}
where $d_i$ ($i=1,2,...,n)$ is given by formula (2.2).
\end{theo}



Theorem 2.2 implies
\begin{cor}
The function $g(m)$ is a polynomial, and is divisible by $m$.
 \end{cor}
 
 Corollary 2.3 implies the well-known fact that the function $S_n(m)=\sum_{x=1}^{m}x^{n}$ is a polynomial \cite{F}. 
\vskip+2mm
Theorem 2.2 implies the explicit formula for $S_n(m)$, for any natural $n$. Note that it does not contain Bernoulli numbers.

\begin{theo} For any natural $n$, one has the equality
 \begin{equation}
 \sum_{x=1}^{m}x^{n}=(-1)^{n}\sum_{i=1}^{n}a_im(m+1)(m+2)...(m+i),
 \end{equation}
where  
\begin{equation}
a_i=\frac{1}{i+1}\sum_{k=1}^{i}\frac{(-1)^{k}k^{n}}{k!(i-k)!},
\end{equation}
 $(i=1,2,...,n)$.
\end{theo} 
 \vskip+2mm
 From formula (2.7) it follows the well-known fact that 
 the polynomial $$S_n(m)=\sum_{x=1}^{m}x^{n}$$ is divisible by $m(m+1)$, for any $n\in \mathbb{N}$ \cite{F}. 
 
 Formula (2.8) implies that 
 \begin{equation}
 a_1=-\frac{1}{2}.
 \end{equation}
 Taking formula (2.7) and the fact that the principal term of the polynomial $S_n(m)$ is $\frac{1}{n+1}m^{n+1}$ \cite{K}, \cite{BSh},  the calculation of $a_n$ can be simplified. Namely, we obtain that
  \begin{equation}
 a_n=\frac{(-1)^n}{n+1}.
 \end{equation}

 From Theorem 2.4, formulas (2.9) and (2.10), we obtain
 
 \begin{theo}
 Let $n\geq 3$. Then one has 
 $$\sum_{x=1}^{m}x^{n}=(-1)^{n}m(m+1)(-\frac{1}{2}+\sum_{i=2}^{n}a_i(m+2)(m+3)...(m+i)),$$
 where $a_i$ is given by formula (2.8) for $i=2,3,...,n-1$, while $a_n$ is given by (2.10).
 \end{theo}
 
 Finally note that since formula (2.8) implies that 
 $$a_n=\frac{1}{(n+1)!}\sum_{k=1}^{n}(-1)^{k}C_n^k k^n,$$
 formula (2.10) implies the following combinatorial identity.
 
 \begin{prop}
 For any natural number $n$, one has
 $$\sum_{k=1}^{n}(-1)^{k}C_n^k k^n=(-1)^nn!.$$
 \end{prop}

\vskip+3mm

\textit{Author's address:}

\textit{Eteri Samsonadze,
Retd., I. Javakhishvili Tbilisi State University,}

\textit{1 Tchavchavadze Av., Tbilisi, 0179, Georgia}, 

\textit{e-mail: eteri.samsonadze@outlook.com} 
\end{document}